\documentclass[11pt,bezier]{article}
\usepackage{amsmath, amssymb, amsthm, amsfonts,tikz,url}
\usepackage{color,soul}
\usepackage{changes}

\newtheorem{thm}{Theorem}
 \newtheorem{conj}[thm]{Conjecture}
 \newtheorem{lem}[thm]{Lemma}
 
\newcommand{\x}{{\bf x}}
\newcommand{\y}{{\bf y}}
\newcommand{\1}{{\bf 1}}
\newcommand{\0}{{\bf 0}}
\newcommand{\rk}{{\rm rank}}
\newcommand{\col}{{\rm Col}}
\newcommand{\nul}{{\rm Nul}}
\newcommand{\F}{{\cal F}}
\newcommand{\g}{\gamma}

\newcommand{\ep}{\epsilon}
\title{\bf\Large  Maximal graphs with respect to rank}

\author{\large H. Esmailian$^{\,\rm a}$ \quad  E. Ghorbani$^{\,\rm a,}$\thanks{Corresponding author} \quad  S. Hossein Ghorban$^{\,\rm b}$ \quad   G.B. Khosrovshahi$^{\,\rm c}$\\[.3cm]
{\sl $^{\rm a}$Department of Mathematics, K. N. Toosi University of Technology,}\\
{\sl P. O. Box 16765-3381, Tehran, Iran}\\
{\sl $^{\rm b}$School of Computer Science, Institute for Research in Fundamental
Sciences (IPM),}\\{\sl P. O. Box 19395-5746, Tehran, Iran }\\
{\sl $^{\rm c}$School of Mathematics, Institute for Research in Fundamental
Sciences (IPM),}\\{\sl P. O. Box 19395-5746, Tehran, Iran }
\\[.3cm]{
{\tt\{h.esmailian,\,e\_ghorbani,\,s.hosseinghorban,\,rezagbk\}@ipm.ir}}}
\begin{document}
\maketitle
\vspace{5mm}
\begin{abstract}
The rank of a graph is defined to be the rank of its adjacency matrix.
A graph is called  reduced if it has no isolated vertices and no two vertices  with the same set of neighbors.
A reduced graph $G$ is said to be maximal if any reduced graph containing $G$ as a proper induced subgraph has a higher rank.
The main intent of this paper is to present some results on maximal graphs. First, we introduce a characterization of maximal trees (a reduced tree is a maximal tree if it is not a proper subtree of a reduced tree with the same rank).
Next, we  give a near-complete characterization of maximal `generalized friendship graphs.' Finally, we present an enumeration of all  maximal graphs with ranks $8$ and $9$. The ranks up to $7$ were already done by Lepovi\'c (1990), Ellingham (1993), and Lazi\'c (2010).
\vspace{3mm}\\

\noindent {\bf Keywords:}  Rank,  Maximal graph, Maximal tree, Generalized friendship graph  \\[.1cm]
\noindent {\bf AMS Mathematics Subject Classification\,(2010):}   05C50, 05C05, 15A03.
\end{abstract}
\vspace{5mm}
\section{Introduction}

Let $G$ be a simple graph with vertex set $\{v_1, \ldots , v_n\}$.  The {\em adjacency matrix} of $G$ is an $n \times  n$
matrix $A(G)$ whose $(i, j)$-entry is $1$ if $v_i$ is adjacent to $v_j$ and  $0$ otherwise.
The number of vertices of $G$ is the {\em order} of $G$.
The {\em rank} of  $G$, denoted by $\rk(G)$,  is the  rank of $A(G)$.
We say that  $G$ is {\em reduced} if it has no isolated vertex  and  no two vertices  with the same set of neighbors.
In the literature, reduced graphs are also known as {\em canonical graphs}  \cite{laz,lep,tor1,tor2}.
There are only finitely many reduced graphs of rank $r$ since the order of such graphs are at most $2^r -1$  \cite{ack,ell}.
A natural question is:
 what is the maximum order  of
a reduced graph with a given rank $r$.
Kotlov and   Lov\'asz \cite{kl} answered this question asymptotically. They proved that the maximum order of such graph is $O(2^{r/2})$.
Later on, Akbari, Cameron, and  Khosrovshahi \cite{ack} made the following conjecture on the exact value of the maximum order.

\begin{conj}\label{conj}  For every integer $r\geqslant2$, the maximum order of any  reduced  graph of  rank $r$
is equal to
$$n(r)=\left\{\begin{array}{ll}
   2\cdot2^{r/2}-2 & \hbox{if $r$ is even}, \\
    5\cdot2^{(r-3)/2}-2 & \hbox{if $r>1$ is odd}.
\end{array}\right.$$
\end{conj}
Ghorbani,  Mohammadian, and  Tayfeh-Rezaie \cite{gmt1} showed  that if   Conjecture~\ref{conj} is not true, then there would be a  counterexample  of rank  at most $47$. They also showed that the order of every  reduced graph of  rank $r$  is at most $8n(r)+14$. The maximum order of graphs with a fixed rank within the families of trees, bipartite graphs and triangle-free graphs were determined  \cite{gmt,gmt2}.

In  this paper, we consider maximal graphs with respect to rank.
A reduced graph $G$ is called {\em maximal} if it is not a proper induced subgraph of a reduced graph with the same rank as $G$. In other words, $G$ is maximal if for any reduced graph $H$ such that $G$ is obtained by removing a vertex form $H$, one has $\rk(H)>\rk(G)$.
Note that the graphs attaining the maximum order in Conjecture~\ref{conj} would be necessarily maximal.
Maximal graphs can also be considered within a specific family of graphs. Let $\F$ be a given family of graphs.
A reduced graph $G\in\F$ is called {\em maximal within $\F$} if for any reduced graph $H\in\F$ such that $G$ is obtained by removing a vertex form $H$, we have $\rk(H)>\rk(G)$. In the classification of graphs with respect to the rank,  maximal graphs are central objects, since any reduced graph of rank $r$ is an induced subgraph of a maximal graph with rank $r$. This remains valid for maximal graphs within a specific family of graphs.
In the paper, we consider both maximal graphs in its general sense (in Sections~\ref{sec:graphs} and \ref{sec:rank8})  and maximal graphs within the family of trees (in Section~\ref{sec:tree}).

In~\cite{gmt}, a characterization of maximal trees (i.e. maximal graphs within the family of trees) is reported. In Section~\ref{sec:tree}, we show that the characterization of \cite{gmt} is not exhaustive and we present a complete characterization of maximal trees. In fact, there is one more construction of such trees which is missing in \cite{gmt}. Ellingham~\cite{ell} presented some families of maximal graphs and  characterized  maximal friendship graphs. In Section~\ref{sec:graphs}, we present a near-complete characterization of maximal `generalized friendship graphs.' All maximal graphs of rank up to 7 were presented in \cite{ell} and independently in \cite{tor1,tor2,lep,laz}.  We continue this line of work by constructing all  maximal graphs of rank $8$ and $9$. A report on this construction is given in Section~\ref{sec:rank8}.

\section{Maximal trees}\label{sec:tree}
A vertex with degree one is called {\em pendant}. A vertex adjacent to a pendant vertex is said to be {\em pre-pendant}.
A tree is reduced if it has no two pendant vertices with the same neighbor.
A {\em maximal tree} is a tree which is maximal within the family of trees, i.e.
  it is not a proper subgraph of a reduced tree with the same rank.

In \cite{gmt}, a characterization of maximal trees was reported as follows: every maximal tree $T$ of rank $r\geqslant4$ is obtained from a maximal tree $T'$ of rank $r-2$ in one of the following two ways:
\begin{itemize}
\item[(i)] attaching a vertex of a $P_2$ to a vertex of $T'$ of rank $r-2$ which is neither pendant nor pre-pendant;
\item[(ii)] attaching a pendant vertex of a $P_3$ to a pre-pendant vertex of $T'$ with rank $r-2$;
\end{itemize}
where $P_n$  denotes the path graph of order $n$.
We claim that the above characterization is not exhaustive. To see this, consider the tree $T$ of Figure~\ref{example}.
\begin{figure}
\centering\begin{tikzpicture}
\draw [line width=.5pt] (-3,0)-- (4,0);
\draw [line width=.5pt] (1,0)-- (1,-1);
\draw [line width=.5pt] (0,0)-- (0,-1);
\draw [fill=black] (-3,0) circle (3pt);
\draw (-3,0.3) node {$\alpha$};
\draw [fill=black] (-2,0) circle (3pt);
\draw (-2,0.3) node {$0$};
\draw [fill=black] (-1,0) circle (3pt);
\draw (-1,0.3) node {$-\alpha$};
\draw [fill=black] (0,0) circle (3pt);
\draw (0,0.3) node {$0$};
\draw [fill=black] (0,-1) circle (3pt);
\draw(0,-1.3) node {$\alpha$};
\draw [fill=black] (1,0) circle (3pt);
\draw (1,0.3) node {$0$};
\draw [fill=black] (1,-1) circle (3pt);
\draw (1,-1.35) node {$\beta$};
\draw [fill=black] (2,0) circle (3pt);
\draw (2,0.3) node {$-\beta$};
\draw [fill=black] (3,0) circle (3pt);
\draw (3,0.3) node {$0$};
\draw [fill=black] (4,0) circle (3pt);
\draw (4,.3) node {$\beta$};
\end{tikzpicture}
\caption{A maximal tree which is not obtained by (i) or (ii).}\label{example}
\end{figure}
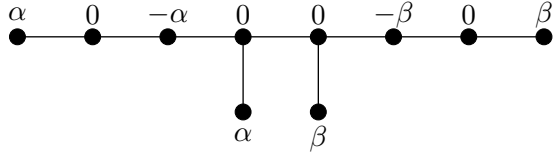
For any given real numbers $\alpha,\beta$, the vector shown on the vertices of $T$ forms a null vector of $A(T)$. (Observe that the components of the given vector on the neighbors of every vertex sum up to $0$.) In fact any null vector of $A(T)$ has this form and thus by Lemma~\ref{nullprependant} (below), $T$ is a maximal tree.
Note that $T$ cannot be obtained by (i). However, it can be obtained by attaching a pendant vertex of a $P_3$ to a pre-pendant vertex of some tree $T'$ which is not maximal. This means that $T$ cannot be constructed by (i) nor by (ii).

In this section, we show that there is one more construction which completes the characterization of maximal trees given in \cite{gmt}.

We denote the column space and the null space of a matrix $M$
  by $\col(M)$ and $\nul(M)$, respectively.
A vertex $v$ of a graph $G$ is called a {\em null vertex} if for every $\x\in\nul(A(G))$, the corresponding component to $v$ is zero. Note that a pre-pendant vertex is always a null vertex. If $S$ is a subset of vertices of $G$, we denote the graph obtained by removing the vertices of $S$ from $G$ by $G-S$. For simplicity, we use $G-v$ for $G-\{v\}$. We denote the degree of a vertex $v$ in a graph $G$ by $d_G(v)$, or by $d(v)$.

The following lemma is well-known and  easy to verify.
\begin{lem}\label{G-{u,v}}
Let $G$ be a graph and $u$ be a pendant vertex of $G$ with the neighbor $v$. Then $\rk(G) = \rk(G-\{u,v\}) + 2$.
\end{lem}

From Lemma~\ref{G-{u,v}} and induction, the following fact can be deduced.
\begin{lem}\label{matching}
The rank of any tree is twice its matching number.
\end{lem}

The following lemma gives a characterization of maximal trees in terms of null vertices.
\begin{lem} [\cite{gmt}]\label{nullprependant}
A reduced tree $T$ is maximal if and only if for every vertex $v$ which is not pre-pendant, $\rk(T) = \rk(T-v)$; or equivalently, $v$ is a null vertex if and only if it is  pre-pendant.
\end{lem}

Now, we  present the main result of this section on the characterization of maximal trees.
\begin{thm}\label{thm:trees}
Every maximal tree $T$ of rank $r\geqslant4$ is obtained from a maximal tree $T'$ of a smaller rank in one of the following three ways:
\begin{itemize}
\item[{\rm (i)}] attaching a vertex of a $P_2$ to a vertex of $T'$ with rank $r-2$ which is neither pendant nor pre-pendant;
\item[{\rm (ii)}] attaching a pendant vertex of a $P_3$ to a pre-pendant vertex of $T'$ with rank $r-2$;
\item[{\rm (iii)}] attaching a pre-pendant vertex of a $P_5$ to a pre-pendant vertex of $T'$ with rank $r-4$ for $r\geqslant8$.
\end{itemize}
\end{thm}

\begin{proof}{
 We first show that any tree resulting from (i)--(iii) is maximal.
 Let $T'$ be a maximal tree and $T$ is obtained by attaching a vertex $v_1$ of a $P_2$ to a vertex $u$ of $T'$. Let $v_2$ be the other vertex of $P_2$.
 In view of Lemma~\ref{G-{u,v}}, $\dim\nul(A(T))=\dim\nul(A(T'))$. We see that any $\x'\in\nul(A(T'))$ can be extended to a $\x\in\nul(A(T))$ by defining $\x(v_1)=0$ and $\x(v_2)=-\x'(u)$. It follows that, besides $v_1$, all other null vertices and also pre-pendant vertices of $T$ and of $T'$ coincide. So by Lemma~\ref{nullprependant}, $T$ is maximal.

 Next, let $T$ be obtained by (ii) from $T'$. Suppose that $v_1,v_2,v_3$ are the vertices of a $P_3$, where $v_1$ is attached to a pre-pendant vertex $u$ of $T'$ and $u'$ is the pendant neighbor of $u$.
From Lemma~\ref{G-{u,v}} it follows that $\rk(T)=\rk(T')+2$ which means $\dim\nul(A(T))=\dim\nul(A(T'))+1$. Let $\{\x'_1, \ldots, \x'_{s-1}\}$ be a basis for  $\nul(A(T'))$. We introduce a basis $\{\x_1, \ldots, \x_s\}$ for  $\nul(A(T))$ as follows. For $1\leqslant i\leqslant s-1$, we extend $\x'_i$  to  $\x_i\in\nul(A(T))$ by defining $\x_i(v_1)=\x_i(v_2)=\x_i(v_3)=0$. Further, let $\x_s$ to be zero on $V(T'-u')$, $\x_s(u') = -\x_s(v_1) = \x_s(v_3) = 1$ and $\x_s(v_2) = 0$. In view of Lemma~\ref{nullprependant}, it turns out that $T$ is a maximal tree. The argument for (iii) is similar to (ii).

 Now, let $T$ be a maximal tree of rank $r\geqslant4$ which is not obtained by (i). We prove that $T$ is obtained by (ii) or (iii).
Note that the only reduced tree of rank $\geqslant4$ and diameter $\leqslant3$ is $P_4$ which is not maximal. So the diameter of $T$ is at least $4$.
Consider a longest path $P$ in $T$ and call its first five vertices from one end $u, v, w, y, z$, respectively. So $u$ is a pendant vertex and $d(v) = 2$. We claim that $w$ is not a pre-pendant vertex. Otherwise, for any vector $\x\in\nul(A(T))$, we have $\x(w) = 0$. Also, since the sum of the components of $\x$ corresponding to the neighbors of $v$ is zero, we have $\x(u) = 0$ which is impossible by Lemma~\ref{nullprependant}. This proves the claim. Furthermore, if $d(w)\geqslant3$, then by Lemmas~\ref{G-{u,v}} and \ref{nullprependant}, $T-\{u,v\}$ would be a maximal tree of rank $r-2$ (since $\nul(A(T-\{u,v\}))$ can be obtained by the restriction of the vectors of $\nul(A(T))$ to $T-\{u,v\}$) which contradicts the assumption on $T$. Thus $d(w) = 2$. We show that $T' = T - \{u, v, w\}$ is a reduced tree of rank $r-2$. Applying Lemmas \ref{G-{u,v}} and \ref{nullprependant}, we find that $\rk(T') = \rk(T-u) - 2 = r - 2$.
To prove that $T'$ is reduced, it suffices to show that $y$ is a pre-pendant vertex in $T$.
Let $M$ be a maximum matching of $T$. If $y$ is not covered by $M$, then $wy\not\in M$. It turns out that $\left(M\setminus\{vw\}\right)\cup\{uv,wy\}$ is a matching of $T$ with  larger size than $M$ which in turn implies that $y$ is covered by every maximum matching of $T$, and so by Lemma~\ref{matching}, $\rk(T-y) = r-2$. From Lemma~\ref{nullprependant}, it follows that $y$ is a pre-pendant vertex of $T$, as desired.
 Hence $T'$ is reduced.  If $T'$ is a maximal tree, then $T$ is obtained by (ii).
 Now, suppose that  $T'$ is not a maximal tree.  Let $p$ be the pendant neighbor of $y$. Recall that $z$ is also a neighbor of $y$.   We show that
\begin{itemize}
\item[{\rm (a)}] $p$ is the only null vertex of $T'$ which is not pre-pendant;
\item[\rm (b)] $z$ is a pre-pendant vertex of $T'$;
\item[\rm (c)] $d_{T'}(y) = 2$;
\item[\rm (d)] $T'' = T' - \{y, p\}$ is a maximal tree of rank $r-4$.
\end{itemize}
The claimed conditions is demonstrated in Figure~\ref{figureTree}.
From (a)--(d) it follows that $T$ is obtained by (iii). So the proof will be completed by verifying (a)--(d) as follows.
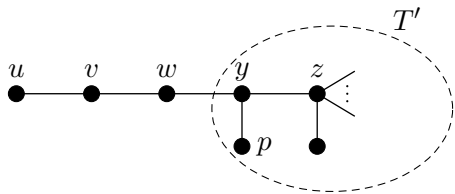
\begin{figure}\centering
\begin{tikzpicture}
\draw (3.2,1) node {$T'$};
\draw [line width=.5pt] (-2,0)-- (2,0);
\draw [line width=.5pt] (2,0)-- (2.5,.3);
\draw [line width=.5pt] (2,0)-- (2.5,-.3);
\draw [line width=.5pt] (1,0)-- (1,-.7);
\draw [line width=.5pt] (2,0)-- (2,-.7);
\draw [fill=black] (0,0) circle (3pt);
\draw (0,0.3) node {$w$};
\draw [fill=black] (-1,0) circle (3pt);
\draw (-1,0.3) node {$v$};
\draw [fill=black] (-2,0) circle (3pt);
\draw(-2,0.3) node {$u$};
\draw [fill=black] (1,0) circle (3pt);
\draw (1,0.3) node {$y$};
\draw [fill=black] (2,0) circle (3pt);
\draw (2,0.3) node {$z$};
\draw [fill=black] (2,-.7) circle (3pt);
\draw [fill=black] (1,-.7) circle (3pt);
\draw (1.3,-.7) node {$p$};
\draw [fill=black] (2.4,0.1) circle (.2pt);
\draw [fill=black] (2.4,0) circle (.2pt);
\draw [fill=black] (2.4,-0.1) circle (0.2pt);
\draw[densely dashed] (2.2,-.2) ellipse (1.6cm and 1.1cm);
\end{tikzpicture}
\caption{The situation of $T$ in Case~(iii).}
\label{figureTree}
\end{figure}
\begin{itemize}
\item[(a)] As $T'$ is not maximal, in view of Lemma~\ref{nullprependant},  $T'$ has at least one non-pre-pendant null vertex. Suppose that $q\neq p$ is a null vertex of $T'$ which is not pre-pendant. Let $\{\x'_1, \ldots, \x'_{s-1}\}$ be a basis for the null space of $A(T')$. We introduce a basis $\{\x_1, \ldots, \x_s\}$ for the null space of $A(T)$ as follows. For $1\leqslant i\leqslant s-1$, we let $\x_i(a) = \x'_i(a)$ for every $a \in V(T')$ and we set $\x_i(u) = \x_i(v) = \x_i(w) = 0$. Moreover, let $\x_s$  be zero on $V(T'-p)$, $\x_s(u) = -\x_s(w) = \x_s(p) = 1$, and $\x_s(v) = 0$. All $\x_1, \ldots, \x_s$ are zero on $q$ which means that $q$ is a non-pre-pendant null vertex for $T$ which is a contradiction by Lemma~\ref{nullprependant}. Therefore, $p$ is a unique non-pre-pendant null vertex of $T'$.
\item[(b)]  We claim that all the neighbors of $y$, excluding $p$, are pre-pendant. To obtain a contradiction, let  $h$ be a  non-pre-pendant neighbor of $y$. Since $p$ is the only non-pre-pendant null vertex of $T'$, $h$ is not a null vertex and thus there is a vector $\x\in\nul(A(T'))$ such that $\x(h)\ne0$. Let $T''$ be the connected component of $T'-y$ containing $h$. We define the vector $\y$ on $V(T)$ such that $\y(a) = 2\x(a)$ for $a \in V(T'')$, $\y(p) = -\x(h)$, and $\y(b) = \x(b)$ for the remaining vertices $b$ of $T'$.
      Clearly, $\y$ belongs to  $\nul(A(T'))$ with $\y(p) \neq 0$. So $p$ is not a null vertex which is a contradiction. Therefore, excluding $p$ all the neighbors of $y$ (including $z$) are pre-pendant.
 \item[(c)]  We establish this claim by a contradiction. Assume that $d_{T'}(y)=k\geqslant3$, and $T'_1,\ldots,T'_k$ are the components of $T'-y$.
 If for at least two $j$'s, $T'_j$ contains a vertex in distance $\geqslant4$ from $y$, then we have a path longer than $P$ in $T$ which is a contradiction.
 So, for some $j$, any pendant vertex $q$ of $T'_j$ have distance $\ell\leqslant3$ from $y$. If $\ell=3$, let $Q=qq_1q_2y$ be the path between $q$ and $y$.
 The vertex $q_1$ is pre-pendant and thus a null vertex. The vertex $q_2$ is a neighbor of $y$ and by (b), it is pre-pendant and hence a null vertex.  Now, since $Q$ is a longest path between a vertex of $T'_j$ and $y$, we have $d_T(q_1)=2$. As the two neighbors of $q$ are null, it follows that $q$ is also null which is a contradiction. If $\ell=2$, then we consider  $Q=qq_1y$. Since $y$ is a pre-pendant vertex,  $y$ is a null vertex. Similarly, we have $d_T(q_1)=2$. Thus $q$ is a null vertex which is a contradiction.
It turns out that $k=2$. 
\item[(d)] Lemma \ref{G-{u,v}} implies that $\rk(T'')=r-4$. As $y$ and $p$ are null vertices of $T'$, $\nul(A(T''))$ can be obtained by the restriction of any vector of $\nul(A(T'))$ to $T''$. From (a), it follows that every non-pre-pendant vertex of $T''$ is not a null vertex and so by Lemma~\ref{nullprependant},  $T''$ is a maximal tree.
\end{itemize}
The proof is now complete.}\end{proof}

For an illustration of how maximal trees with rank up to 8 can be constructed by Theorem~\ref{thm:trees}, see Table~\ref{tab:maxtree}.

\begin{table}[h!]\label{tab:maxtree}
\centering
\begin{tabular}{cccc}
\hline
Rank & Maximal trees \\ \hline \vspace{-.35cm} \\
2&
	\begin{tikzpicture}[scale=.8]
	\draw [line width=.5pt] (-3,0)-- (-2,0);	
	\draw [fill=black] (-3,0) circle (2.5pt);
	\draw [fill=black] (-2,0) circle (2.5pt);	
	\end{tikzpicture}
\\ \hline   \vspace{-.35cm} \\
4&
\begin{tikzpicture}[scale=.8]
\draw [line width=.5pt] (0,0)-- (1,0);
\draw [line width=.7pt,densely dashed] (1,0)-- (2,0);
\draw [line width=.5pt] (2,0)-- (3,0);
\draw [line width=.5pt] (3,0)-- (4,0);
\draw [fill=black] (0,0) circle (2.5pt);
\draw [fill=black] (1,0) circle (2.5pt);
\draw (2,0) circle (3pt);
\draw (3,0) circle (3pt);
\draw (4,0) circle (3pt);
\end{tikzpicture}
\\ \hline  \vspace{-.2cm} \\
6&
$\begin{array}{ccc}
		\begin{tikzpicture}[scale=.75]
		\draw [line width=.5pt] (0,0)-- (1,0);
		\draw [line width=.5pt] (1,0)-- (2,0);
		\draw [line width=.5pt] (2,0)-- (3,0);
		\draw [line width=.5pt] (3,0)-- (4,0);
		\draw [line width=.7pt,densely dashed] (2,0)-- (2,1);
		\draw [line width=.5pt] (2,1)-- (3,1);
		\draw [fill=black] (0,0) circle (2.5pt);
		\draw [fill=black] (1,0) circle (2.5pt);
		\draw [fill=black] (2,0) circle (2.5pt);
		\draw [fill=black] (3,0) circle (2.5pt);
		\draw [fill=black] (4,0) circle (2.5pt);
		\draw (2,1) circle (3pt);
		\draw (3,1) circle (3pt);
		\end{tikzpicture}&&
		\begin{tikzpicture}[scale=.75]
		\draw [line width=.5pt] (0,0)-- (1,0);
		\draw [line width=.5pt] (1,0)-- (2,0);
		\draw [line width=.5pt] (2,0)-- (3,0);
		\draw [line width=.5pt] (3,0)-- (4,0);
		\draw [line width=.5pt] (1,1)-- (2,1);
		\draw [line width=.5pt] (2,1)-- (3,1);
		\draw [line width=.5pt] (3,1)-- (1,1);
		\draw [line width=.7pt,densely dashed] (1,1)-- (1,0);
		\draw [fill=black] (0,0) circle (2.5pt);
		\draw [fill=black] (1,0) circle (2.5pt);
		\draw [fill=black] (2,0) circle (2.5pt);
		\draw [fill=black] (3,0) circle (2.5pt);
		\draw [fill=black] (4,0) circle (2.5pt);
		\draw  (1,1) circle (3pt);
		\draw  (2,1) circle (3pt);
		\draw  (3,1) circle (3pt);
		\end{tikzpicture}
	\end{array}$
\\ \hline  \vspace{-.2cm} \\
8&
$\begin{array}{c}
\begin{array}{ccc}	
		\begin{tikzpicture}[scale=.8]
		\draw [line width=.5pt] (0,0)-- (1,0);
		\draw [line width=.5pt] (1,0)-- (2,0);
		\draw [line width=.5pt] (2,0)-- (3,0);
		\draw [line width=.5pt] (3,0)-- (4,0);
		\draw [line width=.5pt] (2,0)-- (2,1);
		\draw [line width=.5pt] (2,1)-- (3,1);
		\draw [line width=.5pt] (2,-1)-- (3,-1);
		\draw [line width=.7pt,densely dashed] (2,0)-- (2,-1);
		\draw [fill=black] (0,0) circle (2.5pt);
		\draw [fill=black] (1,0) circle (2.5pt);
		\draw [fill=black] (2,0) circle (2.5pt);
		\draw [fill=black] (3,0) circle (2.5pt);
		\draw [fill=black] (4,0) circle (2.5pt);
		\draw [fill=black] (2,1) circle (2.5pt);
		\draw [fill=black] (3,1) circle (2.5pt);
		\draw  (2,-1) circle (3pt);
		\draw  (3,-1) circle (3pt);
		\end{tikzpicture}&&
	\begin{tikzpicture}[scale=.8]
	\draw [line width=.5pt] (0,0)-- (1,0);
	\draw [line width=.5pt] (1,0)-- (2,0);
	\draw [line width=.5pt] (2,0)-- (3,0);
	\draw [line width=.5pt] (3,0)-- (4,0);
	\draw [line width=.5pt] (0,-1)-- (1,-1);
	\draw [line width=.5pt] (1,-1)-- (2,-1);
	\draw [line width=.5pt] (2,-1)-- (3,-1);
	\draw [line width=.5pt] (3,-1)-- (4,-1);
	\draw [line width=.7pt,densely dashed] (1,0)-- (1,-1);
	\draw [fill=black] (0,0) circle (2.5pt);
	\draw [fill=black] (1,0) circle (2.5pt);
	\draw [fill=black] (2,0) circle (2.5pt);
	\draw [fill=black] (3,0) circle (2.5pt);
	\draw [fill=black] (4,0) circle (2.5pt);	
	\draw  (0,-1) circle (3pt);
	\draw  (1,-1) circle (3pt);
	\draw  (2,-1) circle (3pt);
	\draw  (3,-1) circle (3pt);
	\draw  (4,-1) circle (3pt);
	\end{tikzpicture}\end{array} \\ \vspace{-.3cm} \\
\begin{array}{ccc}	
		\begin{tikzpicture}[scale=.8]
		\draw [line width=.5pt] (0,0)-- (1,0);
		\draw [line width=.5pt] (1,0)-- (2,0);
		\draw [line width=.5pt] (2,0)-- (3,0);
		\draw [line width=.5pt] (3,0)-- (4,0);
		\draw [line width=.5pt] (2,0)-- (2,1);
		\draw [line width=.5pt] (2,1)-- (3,1);
		\draw [line width=.5pt] (3,-1)-- (4,-1);
		\draw [line width=.5pt] (4,-1)-- (5,-1);
		\draw [line width=.7pt,densely dashed] (3,0)-- (3,-1);
		\draw [fill=black] (0,0) circle (2.5pt);
		\draw [fill=black] (1,0) circle (2.5pt);
	    \draw [fill=black] (2,0) circle (2.5pt);
	   	\draw [fill=black] (3,0) circle (2.5pt);
		\draw [fill=black] (4,0) circle (2.5pt);
		\draw [fill=black] (2,1) circle (2.5pt);
		\draw [fill=black] (3,1) circle (2.5pt);
		\draw (3,-1) circle (3pt);
		\draw (4,-1) circle (3pt);
		\draw (5,-1) circle (3pt);
		\end{tikzpicture}&	
	\begin{tikzpicture}[scale=.8]
\draw [line width=.5pt] (0,0)-- (1,0);
\draw [line width=.5pt] (1,0)-- (2,0);
\draw [line width=.5pt] (2,0)-- (3,0);
\draw [line width=.5pt] (3,0)-- (4,0);
\draw [line width=.5pt] (1,1)-- (2,1);
\draw [line width=.5pt] (2,1)-- (3,1);
\draw [line width=.5pt] (1,1)-- (1,0);
\draw [line width=.5pt] (3,-1)-- (4,-1);
\draw [line width=.5pt] (4,-1)-- (5,-1);
\draw [line width=.7pt,densely dashed] (3,0)-- (3,-1);
\draw [fill=black] (0,0) circle (2.5pt);
\draw [fill=black] (1,0) circle (2.5pt);
\draw [fill=black] (2,0) circle (2.5pt);
\draw [fill=black] (3,0) circle (2.5pt);
\draw [fill=black] (4,0) circle (2.5pt);
\draw [fill=black] (1,1) circle (2.5pt);
\draw [fill=black] (2,1) circle (2.5pt);
\draw [fill=black] (3,1) circle (2.5pt);
\draw (3,-1) circle (3pt);
\draw (4,-1) circle (3pt);
\draw (5,-1) circle (3pt);
	\end{tikzpicture}&
	\begin{tikzpicture}[scale=.8]	
	\draw [line width=.5pt] (0,0)-- (1,0);
	\draw [line width=.5pt] (1,0)-- (2,0);
	\draw [line width=.5pt] (2,0)-- (3,0);
	\draw [line width=.5pt] (3,0)-- (4,0);
	\draw [line width=.5pt] (1,1)-- (2,1);
	\draw [line width=.5pt] (2,1)-- (3,1);
	\draw [line width=.5pt] (1,1)-- (1,0);
	\draw [line width=.5pt] (1,-1)-- (2,-1);
	\draw [line width=.5pt] (2,-1)-- (3,-1);
	\draw [line width=.7pt,densely dashed] (1,0)-- (1,-1);	
	\draw [fill=black] (0,0) circle (2.5pt);
	\draw [fill=black] (1,0) circle (2.5pt);
	\draw [fill=black] (2,0) circle (2.5pt);
	\draw [fill=black] (3,0) circle (2.5pt);
	\draw [fill=black] (4,0) circle (2.5pt);
	\draw [fill=black] (1,1) circle (2.5pt);
	\draw [fill=black] (2,1) circle (2.5pt);
	\draw [fill=black] (3,1) circle (2.5pt);
	\draw (1,-1) circle (3pt);
	\draw (2,-1) circle (3pt);
	\draw (3,-1) circle (3pt);
	\end{tikzpicture}
\end{array}
\end{array}$\\ \hline
\end{tabular}
\caption{Maximal trees up to rank $8$ and their recursive constructions by Theorem~\ref{thm:trees};   the white vertices demonstrate the paths $P_2$, $P_3$, and $P_5$. }
\label{tab:trees}
\end{table}

\section{Maximal generalized friendship graphs}\label{sec:graphs}

Ellingham \cite{ell} constructed three families of maximal graphs.
One of these, was the family of {\em friendship graphs} $F=F(n)$ defined by
\begin{align*}
V(F)&=\{a,b_1,\ldots, b_n,c_1,\ldots, c_n\},\\
E(F)&=\{ab_i,ac_i,b_ic_i \mid 1 \leqslant i \leqslant n \}.
\end{align*}
We extend this family to the {\em generalized friendship graphs}, denoted by $F(k,m)$, which are the graphs obtained by adding a vertex to $m$ disjoint copies of the complete graph $K_k$, and joining it to all the vertices of the copies of $K_k$. The resulting graph has $mk+1$ vertices. The special case $F(2,m)$ is the friendship graph. Also $F(1,m)$ is the star with $m$ edges which is not reduced and thus is not a maximal graph.  Ellingham proved that:

\begin{thm}[\cite{ell}]\label{thm:F(2,m)} The graph $F(2,m)$ is maximal if and only if  $m$ is a square-free integer.
\end{thm}

Our goal in this section is to extend this result to the generalized friendship graphs.
 We start with the following useful lemma.

\begin{lem} [\cite{bbddm}]\label{adding vertex}
	Let $B$ be a symmetric matrix and
	$$A=\left(\begin{array}{ccc|c}
	&  &  &  \\
	& B &  & \y \\
	&  &  &  \\ \hline
	& \y^\top &  & b
	\end{array}\right).$$
	\begin{itemize}
		\item[\rm(i)] If $\y\not\in\col(B)$, then $\rk(A)=\rk(B)+2$.\label{part1}
		\item[\rm(ii)] If $\y\in\col(B)$ with $B\x=\y$ and $b\ne\y^\top\x$, then $\rk(A)=\rk(B)+1$.\label{part2}
		\item[\rm(iii)] If $\y\in\col(B)$ with $B\x=\y$ and $b=\y^\top\x$, then $\rk(A)=\rk(B)$.\label{part3}
	\end{itemize}
\end{lem}

\begin{thm}\label{thm:F(k,m)}
Let $k\geqslant2$ and $m\geqslant1$. If $mk$ or  $mk/2$ is a square-free integer, then $F(k,m)$ is a maximal graph.
\end{thm}
\begin{proof}{We fix $k\geqslant2$ and $m\geqslant1$. Let $A$ be the adjacency matrix of $F(k,m)$.
We  write $A$ as
$$A=\left(\begin{array}{c|cccc}
0&\1_k^\top&\1_k^\top&\cdots&\1_k^\top\\ \hline
\1_k &J_k-I_k& O  & \cdots&O \\
\1_k&O&J_k-I_k&  \cdots&O \\
\vdots&\vdots  &&\ddots   &\vdots \\
\1_k&O& O &\cdots & J_k-I_k
\end{array}\right),$$
where $J_k$ is the all $1$'s $k\times k$ matrix and $\1_k$ is the all $1$'s vector of length $k$.
(We remove the subscript $k$ in what follows as it is clear from the context.)
It is straightforward to see that $A$ is invertible with
$$A^{-1}=\frac1d
\left(\begin{array}{c|ccccc}
-a^2&a\1^\top&a\1^\top&a\1^\top&\cdots&a\1^\top\\ \hline
a\1 &bJ-dI& -J & -J & \cdots&-J \\
a\1&-J&bJ-dI& -J & \cdots&-J \\
a\1&-J&-J& bJ-dI & \cdots&-J \\
\vdots&\vdots  &\vdots&&\ddots   &\vdots \\
a\1&-J& -J &-J  &\cdots & bJ-dI
\end{array}\right),$$
where $a=k-1$, $b=mk-1$, and $d=mk(k-1)$.

Let  $\y\in\col(A)$ be a $(0,1)$-vector with $A\x=\y$ and $\x^\top A\x=0$.
We show that $\y=\0$ or $\y$ is a column of $A$. This, in view of  Lemma~\ref{adding vertex}, implies that $F(k,m)$ is a maximal graph.
Let us partition $\x$ and $\y$  as
$$\x =\left(\begin{array}{c}
x_0\\\x_1\\\vdots\\ \x_m
 \end{array}\right)~~\hbox{and}~~
\y =\left(\begin{array}{c}
y_0\\\y_1\\\vdots\\ \y_m
 \end{array}\right),$$
where   $\x_1,\ldots,\x_m,$ $\y_1,\ldots,\y_m$ are vectors of length $k$.
Let $\g_i$ be the number of $1$'s in $\y_i$, that is $\g_i=\y_i^\top\1$, and thus $\y_i^\top J\y_j=\g_i\g_j$.
We have
\begin{align*}
0&=d\x^\top A\x\\
&=d\y^\top A^{-1}\y\\
&=-a^2y_0^2+2ay_0\sum_{i=1}^m\y_i^\top\1+\sum_{i=1}^m\y_i^\top (bJ-dI)\y_i-2\sum_{1\leqslant i<j\leqslant m}\y_i^\top J\y_j\\
&=-(k-1)^2y_0^2+2(k-1)y_0\sum_{i=1}^m\g_i+\sum_{i=1}^m\left((mk-1)\g_i^2-mk(k-1)\g_i\right)-2\sum_{1\leqslant i<j\leqslant m}\g_i\g_j.
\end{align*}
Therefore,
\begin{equation}\label{eq:gamma}
-(k-1)^2y_0^2+\sum_{i=1}^m\left( mk\g_i^2-mk(k-1)\g_i+2(k-1)y_0\g_i\right)-\left(\sum_{i=1}^m\g_i\right)^2=0.
\end{equation}

First, assume that $y_0=0$. Let
$$\ell=\sum_{i=1}^m\g_i.$$
Then from \eqref{eq:gamma}  it follows that
\begin{equation}\label{eq:y0=0}
mk\left(\sum_{i=1}^m\g_i^2-(k-1)\ell\right)-\ell^2=0.
\end{equation}
We claim that $mk\mid\ell$.
From \eqref{eq:y0=0}, it is seen that $mk\mid\ell^2$. Now, if $mk$ is square-free, then we must have $mk\mid\ell$ and we are done.
So let $mk$ be even with $mk/2$ square-free. If $4\nmid mk$, then $mk$ is square-free and again we are done.
Hence we can assume that $4\mid mk$. Thus $8\nmid mk$ since $mk/2$ is square-free. Assume that $mk=4n_0$.
 From \eqref{eq:y0=0}, we have
 $n_0\mid\ell^2$, and since $n_0$ is square-free,  $n_0\mid\ell$.
From \eqref{eq:y0=0}, it is clear that $\ell$ is even. It turns out that $\sum_{i=1}^m\g_i^2$ is also even.
Hence the first term of \eqref{eq:y0=0} is divisible by $8$, and so $8\mid\ell^2$. This yields $4\mid\ell$
which in turn implies that $mk=4n_0\mid\ell$, and the claim follows. Note that $\g_i\leqslant k$ for $i=1,\ldots,m$ and thus $\ell\leqslant mk$.
Hence $\ell=0$ or $\ell=mk$. If $\ell=0$, then $\y=\0$.
If $\ell=mk$, then $\g_1=\cdots=\g_m=k$, and so $\y_1=\cdots=\y_m=\1$, which means that $\y$ is the first column of $A$.

Next, assume that $y_0=1$. From \eqref{eq:gamma}  it follows that
\begin{equation}\label{eq:y0=1}
mk\left(\sum_{i=1}^m\g_i^2-(k-1)\ell\right)-(\ell-(k-1))^2=0.
\end{equation}
It is clear that $mk\mid(\ell-(k-1))^2$.
If $mk$ is square-free, then $mk\mid\ell-(k-1)$. If $mk$ is even with $mk/2$ square-free, then, as in the previous case,
we may suppose that $mk=4n_0$ for some odd integer $n_0$. From \eqref{eq:y0=1}, it is seen that $\ell-(k-1)$ is even.
 It follows that either both $\ell$ and $k-1$ are even or both are odd.
 As the parity of $\sum_{i=1}^m\g_i^2$ and $\ell$ are the same, we see that $\sum_{i=1}^m\g_i^2-(k-1)\ell$ is also even. Hence from \eqref{eq:y0=1} we have that $8\mid(\ell-(k-1))^2$ and so $4\mid\ell-(k-1)$.
Therefore,  $mk=4n_0\mid\ell-(k-1)$. Since $\ell-(k-1)<mk$, it follows that $\ell=k-1$.
Plugging in this into \eqref{eq:y0=1}, we obtain
$$\sum_{i=1}^m\g_i^2=(k-1)^2=\left(\sum_{i=1}^m\g_i\right)^2.$$
This is only possible if exactly one of $\g_i$'s is $k-1$ and the rest are zero.
Consequently, exactly one of the $\y_i$'s is a column of $J-I$, and the rest are $\0$.
This means that $\y$ is the $i$-th column of $A$ for some $2\leqslant i\leqslant mk+1$.
}\end{proof}

Now, we consider the converse of Theorem~\ref{thm:F(k,m)} which holds for $k=2$ by Theorem~\ref{thm:F(2,m)}.
We prove it for $k=3$ in the following theorem. The case $k\geqslant4$ will be discussed afterwards.

\begin{thm}\label{thm:F(3,m)} The graph $F(3,m)$ is maximal if and only if $3m$  or  $3m/2$ is a square-free integer.
\end{thm}
\begin{proof}{If  $3m$ is square-free or $m$ is even with $3m/2$  square-free, by Theorem~\ref{thm:F(k,m)}, $F(3,m)$ is maximal.
The remaining values of $m$ are those divisible by $3$, by $8$, or by a square of an odd integer.
We show that for these values of $m$, $F(3,m)$ is not maximal.
 In view of Lemma~\ref{adding vertex}, proving that $F(3,m)$ is not maximal  amounts to finding a $(0,1)$-vector $\y\in\col(A)$ with $\y^\top A^{-1}\y=0$ such that $\y$ is neither $\0$ nor a column of $A$. Since $A$ is invertible, any $\y$ belongs to $\col(A)$.
To have $\y^\top A^{-1}\y=0$, it suffices to find a solution for \eqref{eq:gamma}, equivalently for \eqref{eq:y0=0} if $y_0=0$ or for \eqref{eq:y0=1} if $y_0=1$.
 Note that the columns of $A$ provide solutions for \eqref{eq:gamma} with $y_0=1$ and exactly one of $\g_1,\ldots,\g_m$ is equal to $2$  and the rest to $0$ or $y_0=0$, and $\g_1=\cdots=\g_m=3$. To complete the proof, we find non-zero solutions other than those coming from the columns of $A$.

For $m=3$ and $y_0=1$,  $\g_1=3,\g_2=\g_3=1$ satisfies \eqref{eq:y0=1}.
In our solutions for other values of $m$,  $y_0=0$. So we consider \eqref{eq:y0=0} with $k=3$.
 Note that $0\leqslant\g_i\leqslant3$.
To simplify  \eqref{eq:y0=0},  let $a_r$ be the number of $\g_i$, $1\leqslant i\leqslant m$, which are equal to $r$ for $r=0,1,2,3$.
Therefore, we may write \eqref{eq:y0=0}  as
\begin{equation}\label{eq:ai}
3m(a_1+4a_2+9a_3)-6m(a_1+2a_2+3a_3)-(a_1+2a_2+3a_3)^2=0.
\end{equation}
We observe that
\begin{equation}\label{eq:m->mt}
 \hbox{if $(m,a_1,a_2,a_3)$ is a sulotion to \eqref{eq:ai}, then so is $(mb,a_1b,a_2b,a_3b)$ for any $b\geqslant1$.}
 \end{equation}
\begin{table}[ht]
\centering
\begin{tabular}{cccc}
\hline
$m$ & $a_1$ & $a_2$ & $a_3$  \\ \hline
6 & 1&1&3\\
8&3&0&3\\
$6t-3$&$t-2$&$t+1$&$t-1$\\
$(2t+1)^2$&$3t$&0&$4t^2+t$\\
\hline
\end{tabular}
\caption{Some solutions to Equation \eqref{eq:ai}}
\label{tab:solutions}
\end{table}

If $m>3$ is divisible by $3$, then $m=6t$ for $t\geqslant1$ or $m=6t-3$ for $t\geqslant2$.
For $m=6$, a solution to \eqref{eq:ai} is given in Table~\ref{tab:solutions}. This together with \eqref{eq:m->mt} gives a solution for any $m=6t$.
For $m=6t-3$ with $t\geqslant2$, a solution to \eqref{eq:ai} is given in Table~\ref{tab:solutions}.
If $m=8t$, then a solution is obtained by the solution for $m=8$ given in Table~\ref{tab:solutions} and employing \eqref{eq:m->mt}.
If $m$ is a multiple of a square of odd integer $(2t+1)^2$, again a solution is obtained from Table~\ref{tab:solutions} and \eqref{eq:m->mt}.
}\end{proof}

Finally, we show that if $m$ is large enough in terms of $k$, then the converse of Theorem~\ref{thm:F(k,m)} holds, that is there are no maximal graphs $F(k,m)$ besides those given in Theorem~\ref{thm:F(k,m)}.

\begin{thm}\label{thm:nonMaximalF(k,m)}
 Let $k\geqslant2$ and $m\geqslant1$.
 If $mk$ is divisible by a square of an odd integer or divisible by $8$, and $m\geqslant\frac{(5k^2-19k+20)^2}{4k}$, then $F(k,m)$ is not maximal.
\end{thm}
\begin{proof}{ For $k=2,3$, the result follows from Theorems~\ref{thm:F(2,m)} and \ref{thm:F(3,m)}. So we assume that $k\geqslant4$.

Similar to the proof of Theorem~\ref{thm:F(3,m)}, our goal is to find solutions to \eqref{eq:y0=0} with $y_0=0$.
 Note that columns of $A$ provide the (trivial) solution
 $\g_1=\cdots=\g_m=k$  to \eqref{eq:y0=0}. To complete the proof, we find non-zero solutions other than this trivial one.

By the assumption, we may write  $mk=cq^2$ for some positive integers $c,q$, where either $q$ is odd, or $q=2$ and $c$ is even.
This in turn implies that whenever $mk$ is even, then $c$ is also even.
If $(\g_1,\ldots,\g_m)$ is a solution to \eqref{eq:y0=0},  then $mk$ divides $\ell^2$. So we will look for a solution with $\ell=cq$.
We observe that if $(\g_1,\ldots,\g_m)$ satisfies
\begin{align*} \sum_{i=1}^m\g_i&= cq,  \\
 \sum_{i=1}^m\g_i^2&=(k-1)cq+c,
 \end{align*}
then it is a solution for \eqref{eq:y0=0}.
 We will show that there is a solution containing only $0$'s, $1$'s, $2$'s, $(k-1)$'s, and $k$'s, i.e.
\begin{equation}\label{eq:sys}
\left\{\begin{array}{l}
  u+2v+(k-1)w+kt=cq,\\
 u+4v+(k-1)^2w+k^2t=(k-1)cq+c,
 \end{array}\right.
 \end{equation}
where $u,v,w,t$ are the multiplicities of $1$'s, $2$'s, $(k-1)$'s, $k$'s, respectively.
Solving \eqref{eq:sys} in $w$ and $t$, yields
\begin{equation}\label{eq:w,t}w:=\frac{c(q-1)-(k-1)u-2(k-2)v}{k-1},\quad t:=\frac{c+(k-2)u+2(k-3)v}{k}.\end{equation}
It follows that \eqref{eq:sys} has an integer solution whenever
\begin{align*}
(k-1)&\mid c(q-1)-(k-1)u-2(k-2)v,\\
k&\mid c+(k-2)u+2(k-3)v,
\end{align*}
that is
\begin{equation}\label{eq:u,v}
\left\{\begin{array}{l} 2v\equiv-c(q-1)\pmod{k-1},\\
2u+6v\equiv c\pmod k.\end{array}\right.
\end{equation}

If $k$ is even, then, as noted above, $c$ is also even.
Therefore, we have the following solution for \eqref{eq:u,v}:
$$v:=-\frac{c(q-1)}{2}\pmod{k-1},\quad u:=\frac{c}{2}-3v\pmod{k/2}.$$
For odd $k$, either $q$ is odd, or $q=2$ in which case $c$ is even. Hence $c(q-1)/2$ is an integer and $c/2$ exists mod $k$. Thus the following gives a solution for \eqref{eq:u,v}:
$$v:=-\frac{c(q-1)}{2}\pmod{(k-1)/2},\quad u:=\frac{c}{2}-3v\pmod k.$$
From \eqref{eq:w,t}, it  follows that $t$ is always positive.
 Further, we have either
 \begin{equation}\label{ineq:u,v}
    0\leqslant u\leqslant k/2-1,\, 0\leqslant v\leqslant k-2,\quad\hbox{or}\quad 0\leqslant u\leqslant k-1,\, 0\leqslant v\leqslant (k-3)/2.
 \end{equation}
 Hence $(k-1)u+2(k-2)v$ is at most $(k/2-1)(5k-9)$ for $k\geqslant4$. It turns out that $w\geqslant0$ since
 $$c(q-1)-(k-1)u-2(k-2)v\geqslant\left(\sqrt{mk}-1\right)-(k/2-1)(5k-9)\geqslant0,$$
 where the last inequality holds for $m\geqslant\frac{(5k^2-19k+20)^2}{4k}$.
 It remains to  verify that $u+v+w+t<m$:
    from the first equation of \eqref{eq:sys},
    $$w+t<\frac{cq}{k-1}=\frac{mk}{q(k-1)}\leqslant\frac{mk}{2(k-1)}\leqslant\frac{2m}{3},$$
and from \eqref{ineq:u,v},
$$u+v\leqslant\frac{3k-5}{2}<\frac{m}{3}.$$
Consequently, we obtain a solution of \eqref{eq:y0=0} different from the trivial one.
}\end{proof}

We expect that the condition on $m$ in Theorem~\ref{thm:nonMaximalF(k,m)} can be improved by considering solutions of \eqref{eq:y0=1}. However, it cannot be removed completely. As a matter of fact, in many cases, the assertion does not hold when $m$ is small. By a computer search, we found all the solutions of \eqref{eq:y0=0} and \eqref{eq:y0=1} for $k\leqslant15,m\leqslant100$. As a result, we come up with several couples $(m,k)$ such that $mk$ is divisible by $8$ or by a square of an odd integer but $F(k,m)$ is maximal; see Table~\ref{tab:MaximalF(k,m)}.
\begin{table}[ht]
\centering
\begin{tabular}{cc}
\hline
$k$ & $m$  \\ \hline
$4$ & $2$\\
$5$ &  | \\
$6$ & $3,4$\\
$7$ & $8,9$\\
$8$ & $2,3,4,5,9$\\
$9$ & $2,3,5,6,7,8$\\
\hline
\end{tabular}
\qquad
\begin{tabular}{cc}
\hline
$k$ & $m$  \\ \hline
$10$ & $4,5,8,9$\\
$11$ & $8,9,16,18$\\
$12$ & $2,3,4,6,8,9,10$\\
$13$ & $8,9,16,18$\\
$14$ & $4,8,9,12,16,18$\\
$15$ & $3,5,6,8,9,10,12,16,18$\\
\hline
\end{tabular}
\caption{The list of $k\leqslant15,m\leqslant100$ such that $mk$ is divisible by $8$ or by a square of an odd integer yet $F(k,m)$ is maximal}
\label{tab:MaximalF(k,m)}
\end{table}

In the next theorem, under certain conditions, we prove the fact suggested by Table~\ref{tab:MaximalF(k,m)} for $m\leqslant12$.
\begin{thm} If $mk=8q$ with $m\leqslant12$, $k\geqslant11$, and $q$ a square-free odd integer, then  $F(k,m)$ is a maximal graph.
\end{thm}
\begin{proof}{For $k=11,12$, the result follows from Table~\ref{tab:MaximalF(k,m)}. So we may assume that $k\geqslant13$.
It suffices to show that Equations \eqref{eq:y0=0} and \eqref{eq:y0=1} have no non-trivial solutions.
We keep using the notation of the proof of Theorem \ref{thm:F(k,m)}.

We first consider the solutions of \eqref{eq:y0=1}.
Let $\g_1,\ldots,\g_m$ satisfies \eqref{eq:y0=1}.
Then $mk\mid(\ell-k+1)^2$, and since $mk=8q$ with $q$ odd and square-free, we have that $mk/2\mid\ell-k+1$.
Note that $\ell\leqslant mk$, and $\ell=k-1$ only for the trivial solution of \eqref{eq:y0=1}. It follows that
$\ell=mk/2+k-1$.
Let $\ep_i=\g_i-\frac k2$ for $i=1,\ldots,m$. Then we see that
\begin{align}
 \sum_{i=1}^m\ep_i&=k-1,\label{eq:1epsilon}\\
\sum_{i=1}^m\ep_i^2&=\frac{mk^2}4-\frac{mk}4-k+1.\label{eq:1epsilon^2}
\end{align}
Since $0\leqslant\g_i\leqslant k$, we have $0\leqslant|\ep_i|\leqslant k/2$.
Let
\begin{equation}\label{eq:ep_b}
|\ep_1|,\ldots,|\ep_b|\leqslant\frac k2-1, \quad |\ep_{b+1}|=\cdots=|\ep_m|=\frac k2.
\end{equation}
From \eqref{eq:1epsilon^2}, it follows that $b(k-1)\leqslant mk/4+k-1$ which implies that $b\leqslant\left\lfloor\frac{mk}{4(k-1)}\right\rfloor+1$.
Thus, as  $k\geqslant13$, we have $b\leqslant3$ for $m\leqslant11$,  and $b\leqslant4$ for $m=12$.

First, let $b=1$, that is $\ep_2=\cdots=\ep_m=\pm k/2$. Hence, $\ep_2+\cdots+\ep_m=j k/2$ for some integer $j\equiv m-1\pmod2$. It turns out that \eqref{eq:1epsilon} holds only if $\ep_1=-1$ or $k/2-1$. If $\ep=-1$, then $j$ must be even which means that $m$ must be odd. Now from \eqref{eq:1epsilon^2}, we have $1+(m-1)k^2/4=mk^2/4-mk/4-k+1$ which implies that
$k=m+4$. So $k$ must be odd and so is $mk$, a contradiction. If  $\ep=k/2-1$, then  \eqref{eq:1epsilon^2} cannot hold, again a contradiction.

Next, let $b=2$. So, by \eqref{eq:1epsilon^2},
\begin{equation}\label{eq:=ep1+ep2}
    \ep_1^2+\ep_2^2=k^2/2-mk/4-k+1.
\end{equation}
We claim that $m$ must be even. Otherwise, $k=8q'$ for some odd $q'$, and so the right hand side of \eqref{eq:=ep1+ep2} is an odd integer.
It also turns out that both $\ep_1$ and $\ep_2$ are  integers: one odd and the other one even. So
$-2mq'+1\equiv\ep_1^2+\ep_2^2\equiv1,5\pmod8$. This implies that $m$ is even, as desired. It follows that
 $\ep_3+\cdots+\ep_m=jk$ for some integer $j$. From \eqref{eq:1epsilon}, then it follows that $\ep_1+\ep_2=-1$ or $k-1$.
 As $\ep_1$ and $\ep_2$ are at most $k/2-1$, the second option is not possible. So $\ep_1+\ep_2=-1$.
 Assume that $\ep_1\geqslant\ep_2$.
 If $\ep_1\leqslant k/2-3$, then $\ep_2\geqslant 2-k/2$, and thus
 $\ep_1^2+\ep_2^2\leqslant k^2/2-5k+13$. On the other hand, by \eqref{eq:=ep1+ep2} and since $m\leqslant12$, we have
 $\ep_1^2+\ep_2^2\geqslant k^2/2-4k+1$. So we must have $-4k+1\leqslant-5k+13$ which does not hold for $k\geqslant13$.
 It follows that $\ep_1= k/2-2$ and $\ep_2=1-k/2$, and so $\ep_1^2+\ep_2^2= k^2/2-3k+5$. From \eqref{eq:=ep1+ep2}, we have
 $-2k+5=-mk/4+1$ which implies that $k=q+2$, that is $k$ is odd. So $8\mid m$ and thus $m=8$ which leads to a contradiction.

Now, let $b=3$, so $m\geqslant8$. By \eqref{eq:1epsilon^2},
\begin{equation}\label{eq:=ep1+ep2+ep3}
\ep_1^2+\ep_2^2+\ep_3^2=3k^2/4-mk/4-k+1.
\end{equation}
From \eqref{eq:1epsilon} and \eqref{eq:ep_b}, we see that $\ep_1+\ep_2+\ep_3=jk/2-1$ for some $j\in\{0,\pm1,\pm2\}$.
If $|\ep_1|=|\ep_2|=|\ep_3|=k/2-1$, then $\ep_1^2+\ep_2^2+\ep_3^2=3k^2/4-3k+3.$ So
$-2k+3=-mk/4+1\leqslant-2k+1$, a contradiction.
It turns out that  at least two of the $|\ep_1|,|\ep_2|,|\ep_3|$ are less than $k/2-1$. So
$$\ep_1^2+\ep_2^2+\ep_3^2\leqslant(k/2-1)^2+2(k/2-2)^2=3k^2/4-5k+9.$$
As $m\leqslant12$, the right hand side of \eqref{eq:=ep1+ep2+ep3} is at least $3k^2/2-4k+1$. It follows that
$-5k+9\geqslant-4k+1$ which holds only for $k\leqslant8$.

Finally, let $b=4$. This is only possible for $m=12$. By \eqref{eq:1epsilon^2},
$$\ep_1^2+\ep_2^2+\ep_3^2+\ep_4^2=k^2-4k+1.$$
From \eqref{eq:1epsilon}, we see that $\ep_1+\ep_2+\ep_3+\ep_4=jk/2-1$ for some integer $j$.
It turns out that not all of the $|\ep_1|,|\ep_2|,|\ep_3|,|\ep_4|$ can be $k/2-1$. So
$$\ep_1^2+\ep_2^2+\ep_3^2+\ep_4^2\leqslant3(k/2-1)^2+(k/2-2)^2=k^2-5k+7.$$
It follows that $-5k+7\geqslant-4k+1$, which holds only for $k\leqslant6$.

Now, we deal with the solutions of  \eqref{eq:y0=0}.
Let $\g_1,\ldots,\g_m$ satisfies \eqref{eq:y0=0}.
Then $mk\mid\ell^2$, and since $mk=8q$ with $q$ odd and square-free, we have that $mk/2\mid\ell$.
Note that $\ell=0$ and $\ell=mk$ only hold for the trivial solution $\g_1=\cdots=\g_m=0$ and $\g_1=\cdots=\g_m=k$, respectively. Therefore,
$\ell=mk/2$. Let $\ep_i=\g_i-\frac k2$ for $i=1,\ldots,m$. Then
\begin{align}
 \sum_{i=1}^m\ep_i&=0,\label{eq:0epsilon}\\
\sum_{i=1}^m\ep_i^2&=\frac{mk^2}4-\frac{mk}4.\label{eq:0epsilon^2}
\end{align}
Let $\ep_1,\ldots,\ep_b$ be as in \eqref{eq:ep_b}. From \eqref{eq:0epsilon^2}, it follows that $b(k-1)\leqslant mk/4$ which implies that $b\leqslant\left\lfloor\frac{mk}{4(k-1)}\right\rfloor$. Thus, we have $b\leqslant2$ for $m\leqslant11$  and $b\leqslant3$ for $m=12$.

With $b=1$, \eqref{eq:0epsilon} can be satisfied
only if $\ep_1=\ep_2+\cdots+\ep_m=0$. In this case, \eqref{eq:0epsilon^2} can be satisfied only if $m=k$ which is not possible since $m\leqslant12<k$.  Next, let $b=2$. So $m\geqslant8$.
We have $\ep_1+\ep_2=0$ or $\pm k/2$ and $\ep_1^2+\ep_2^2=k^2/2-mk/4$.
First, assume that  $\ep_1+\ep_2=0$.
If $|\ep_1|=k/2-1$, then we have $k^2/2-2k+2=\ep_1^2+\ep_2^2=k^2/2-mk/4\leqslant k^2/2-2k$, which is a contradiction.
Hence $|\ep_1|\leqslant k/2-2$, and so $k^2/2-4k+8\geqslant\ep_1^2+\ep_2^2=k^2/2-mk/4\geqslant k^2/2-3k$, which holds only for
$k\leqslant8$.
Second, with no loss of generality, we can assume that $\ep_1+\ep_2=k/2$.
In view of \eqref{eq:ep_b}, both $\ep_1$ and $\ep_2$ must be positive.
So $\ep_1^2+\ep_2^2<k^2/4$.
This implies  $k^2/2-mk/4< k^2/4$, that is
$m>k$ which is impossible. Finally, let $b=3$. So $m=12$ and $\ep_1^2+\ep_2^2+\ep_3^2=3k^2/4-3k$.
By  \eqref{eq:ep_b},  $\ep_1+\ep_2+\ep_3=jk/2$ for some $j\in\{0,\pm1,\pm2\}$.
It turns out that not all of the $|\ep_1|,|\ep_2|,|\ep_3|$ can be $k/2-1$. So
$$3k^2/4-3k=\ep_1^2+\ep_2^2+\ep_3^2\leqslant2(k/2-1)^2+(k/2-2)^2=3k^2/4-4k+6,$$
 which can be satisfied only for $k\leqslant6$. The proof is now complete.
}\end{proof}

We close this section by a summary of the results on maximality of $F(k,m)$: for integers $k\geqslant2$ and $m\geqslant1$,
\begin{itemize}
\item[(i)] $F(k,m)$ is maximal if $mk$ or $mk/2$ is square-free;
\item[(ii)] the converse of (i) holds for $k=2,3$ or  $m\geqslant\frac{(5k^2-19k+20)^2}{4k}$;
\item[(iii)] $F(k,m)$ is maximal if $mk=8q$ with $k\geqslant11$,  $m\leqslant12$, and $q$ a square-free odd integer.
\end{itemize}
These  provide a near-complete characterization of maximal $F(k,m)$.
We leave the complete characterization as an open problem.

\section{Maximal graphs with small rank}\label{sec:rank8}

In this section we present some statistics  of maximal graphs with small rank.
We start by Table~\ref{tab:r2-5} in which all the maximal graphs with rank at most $5$ are depicted.

\begin{table}[h!]
\centering
\begin{tabular}{cccc}
\hline
Rank &&& Maximal graphs \\ \hline  \vspace{-.4cm}\\
2&&&
\begin{tikzpicture}[scale=.5]
	\draw [line width=.5pt] (1,0)-- (3,0);
	\draw [fill=black] (1,0) circle (5pt);
	\draw [fill=black] (3,0) circle (5pt);
\end{tikzpicture} \vspace{.15cm} \\ \hline \vspace{-.35cm} \\
3&&&
$\begin{array}{c}\begin{tikzpicture}[scale=.5]
\draw [line width=.5pt] (-3,0)-- (-1,0);
	\draw [line width=.5pt] (-1,0)-- (-2,1.72);
	\draw [line width=.5pt] (-2,1.72)-- (-3,0);
	\draw [fill=black] (-3,0) circle (5pt);
	\draw [fill=black] (-1,0) circle (5pt);
	\draw [fill=black] (-2,1.72) circle (5pt);
\end{tikzpicture}\end{array}$ \vspace{.15cm} \\ \hline \vspace{-.35cm} \\
4 &&&
$\begin{array}{cccccc}
\begin{tikzpicture}[scale=.4]
	\draw [line width=.5pt] (4,-2)-- (6,-2);
	\draw [line width=.5pt] (6,-2)-- (7,-0.3);
	\draw [line width=.5pt] (7,-0.3)-- (6,1.5);
	\draw [line width=.5pt] (6,1.5)-- (4,1.5);
	\draw [line width=.5pt] (4,1.5)-- (3,-0.3);
	\draw [line width=.5pt] (3,-0.3)-- (4,-2);
	\draw [line width=.5pt] (6,1.5)-- (4,-2);
	\draw [line width=.5pt] (4,1.5)-- (7,-0.3);
	\draw [line width=.5pt] (3,-0.3)-- (6,-2);
	\draw [fill=black] (4,-2) circle (5pt);
	\draw [fill=black] (6,-2) circle (5pt);
	\draw [fill=black] (7,-0.3) circle (5pt);
	\draw [fill=black] (6,1.5) circle (5pt);
	\draw [fill=black] (4,1.5) circle (5pt);
	\draw [fill=black] (3,-0.3) circle (5pt);
	\end{tikzpicture}&&
	\begin{tikzpicture}[scale=.45]
	\draw [line width=.5pt] (-1,-2)-- (1,-2);
	\draw [line width=.5pt] (1,-2)-- (1.6,-0.1);
	\draw [line width=.5pt] (1.6,-0.1)-- (0,1.08);
	\draw [line width=.5pt] (0,1.08)-- (-1.6,-0.1);
	\draw [line width=.5pt] (-1.6,-0.1)-- (-1,-2);
	\draw [line width=.5pt] (-2.5,-2)-- (-1,-2);
	\draw [line width=.5pt] (-1.6,-0.1)-- (1,-2);
	\draw [fill=black] (-1,-2) circle (5pt);
	\draw [fill=black] (1,-2) circle (5pt);
	\draw [fill=black] (1.6,-0.1) circle (5pt);
	\draw [fill=black] (0,1.08) circle (5pt);
	\draw [fill=black] (-1.6,-0.1) circle (5pt);
	\draw [fill=black] (-2.5,-2) circle (5pt);
	\end{tikzpicture}&&
	\begin{tikzpicture}[scale=.5]
	\draw [line width=.5pt] (-6,-2)-- (-4,-2);
	\draw [line width=.5pt] (-4,-2)-- (-4,0);
	\draw [line width=.5pt] (-4,0)-- (-6,0);
	\draw [line width=.5pt] (-6,0)-- (-6,-2);
	\draw [line width=.5pt] (-4,0)-- (-6,-2);
	\draw [line width=.5pt] (-6,0)-- (-4,-2);
	\draw [fill=black] (-6,-2) circle (5pt);
	\draw [fill=black] (-4,-2) circle (5pt);
	\draw [fill=black] (-4,0) circle (5pt);
	\draw [fill=black] (-6,0) circle (5pt);
	\end{tikzpicture}
	\end{array}$ \vspace{.15cm} \\	\hline \vspace{-.35cm} \\
5 &&&		
$\begin{array}{cccccccc}
\begin{tikzpicture}[scale=.35]	
	\draw [line width=.5pt] (2,6)-- (4,6);
	\draw [line width=.5pt] (4,6)-- (5.40,7.40);
	\draw [line width=.5pt] (5.4,7.4)-- (5.4,9.4);
	\draw [line width=.5pt] (5.40,9.40)-- (4,10.83);
	\draw [line width=.5pt] (4,10.83)-- (2,10.83);
	\draw [line width=.5pt] (2,10.83)-- (0.59,9.40);
	\draw [line width=.5pt] (0.59,9.40)-- (0.59,7.40);
	\draw [line width=.5pt] (0.59,7.40)-- (2,6);		
	\draw [line width=.5pt] (2,10.83)-- (5.40,7.40);
	\draw [line width=.5pt] (2,10.83)-- (0.59,7.40);
	\draw [line width=.5pt] (4,10.83)-- (0.59,7.40);
	\draw [line width=.5pt] (4,10.83)-- (0.59,9.40);
	\draw [line width=.5pt] (5.40,9.40)-- (2,6);
	\draw [line width=.5pt] (0.59,9.40)-- (4,6);
    \draw [line width=.5pt] (5.4,9.40)-- (4,6);
	\draw [line width=.5pt] (2,6)-- (5.40,7.40);	
	\draw [fill=black] (2,6) circle (5pt);
	\draw [fill=black] (4,6) circle (5pt);
	\draw [fill=black] (5.40,7.40) circle (5pt);
	\draw [fill=black] (5.40,9.40) circle (5pt);
	\draw [fill=black] (4,10.83) circle (5pt);
	\draw [fill=black] (2,10.83) circle (5pt);
	\draw [fill=black] (0.59,9.40) circle (5pt);
	\draw [fill=black] (0.59,7.40) circle (5pt);
\end{tikzpicture}
 &&
 \begin{tikzpicture}[scale=.4]
 	\draw [line width=.5pt] (-3,-3)-- (-1,-3);
 	\draw [line width=.5pt] (0,1.39)-- (1.80,0.50);
 	\draw [line width=.5pt] (0,1.39)-- (-1.80,0.50);
 	\draw [line width=.5pt] (0,1.39)-- (1,-3);
 	\draw [line width=.5pt] (0,1.39)-- (-1,-3);
 	\draw [line width=.5pt] (0,1.39)-- (-2.20,-1.40);
 	\draw [line width=.5pt] (1.80,0.50)-- (-1.80,0.50);
 	\draw [line width=.5pt] (1.80,0.50)-- (2.25,-1.40);
 	\draw [line width=.5pt] (1.80,0.50)-- (1,-3);
 	\draw [line width=.5pt] (2.25,-1.40)-- (1,-3);
 	\draw [line width=.5pt] (2.25,-1.40)-- (-2.20,-1.40);
 	\draw [line width=.5pt] (2.25,-1.40)-- (-1.80,0.50);
 	\draw [line width=.5pt] (1,-3)-- (-2.20,-1.40);
 	\draw [line width=.5pt] (1,-3)-- (-1,-3);
 	\draw [line width=.5pt] (-1,-3)-- (-2.20,-1.40);
 	\draw [line width=.5pt] (-1.80,0.50)-- (-2.20,-1.40);
	\draw [fill=black] (-1,-3) circle (5pt);
	\draw [fill=black] (1,-3) circle (5pt);
	\draw [fill=black] (-3,-3) circle (5pt);
	\draw [fill=black] (2.25,-1.40) circle (5pt);
	\draw [fill=black] (1.80,0.50) circle (5pt);
	\draw [fill=black] (0,1.39) circle (5pt);
	\draw [fill=black] (-1.80,0.50) circle (5pt);
	\draw [fill=black] (-2.20,-1.40) circle (5pt);
\end{tikzpicture}
&&
	\begin{tikzpicture}[scale=.4]	
	\draw [line width=.5pt] (-1,-3)-- (1,-3);
	\draw [line width=.5pt] (1,-3)-- (2.25,-1.40);
	\draw [line width=.5pt] (2.25,-1.40)-- (1.80,0.50);
	\draw [line width=.5pt] (1.80,0.50)-- (0,1.39);
	\draw [line width=.5pt] (0,1.39)-- (-1.80,0.50);
	\draw [line width=.5pt] (-1.80,0.50)-- (-2.20,-1.40);
	\draw [line width=.5pt] (-2.20,-1.40)-- (-1,-3);
	\draw [line width=.5pt] (1.80,0.50)-- (-1.80,0.50);
	\draw [line width=.5pt] (-2.20,-1.40)-- (2.25,-1.40);
	\draw [line width=.5pt] (1.80,0.50)-- (-2.20,-1.40);
	\draw [line width=.5pt] (-1.80,0.50)-- (2.25,-1.40);
	\draw [fill=black] (-1,-3) circle (5pt);
	\draw [fill=black] (1,-3) circle (5pt);
	\draw [fill=black] (2.25,-1.40) circle (5pt);
	\draw [fill=black] (1.80,0.50) circle (5pt);
	\draw [fill=black] (0,1.39) circle (5pt);
	\draw [fill=black] (-1.80,0.50) circle (5pt);
	\draw [fill=black] (-2.20,-1.40) circle (5pt);
\end{tikzpicture} &&
\begin{tikzpicture}[scale=.4]	
	\draw [line width=.5pt] (13,-3)-- (15,-3);
	\draw [line width=.5pt] (15,-3)--(16.20,-1.4);
	\draw [line width=.5pt] (16.20,-1.40)-- (15.80,0.5);
	\draw [line width=.5pt] (15.8,0.5)-- (14,1.39);
	\draw [line width=.5pt] (14,1.39)-- (12.20,0.50);
	\draw [line width=.5pt] (12.20,0.50)-- (11.75,-1.4);
	\draw [line width=.5pt] (11.75,-1.4)-- (13,-3);
	\draw [line width=.5pt] (12.20,0.50)-- (13,-3);
	\draw [line width=.5pt] (12.20,0.50)-- (15,-3);
	\draw [line width=.5pt] (12.20,0.50)-- (16.20,-1.40);
	\draw [line width=.5pt] (12.20,0.50)-- (15.80,0.50);
	\draw [line width=.5pt] (11.75,-1.40)-- (15.80,0.50);
	\draw [line width=.5pt] (11.75,-1.40)-- (14,1.39);
	\draw [line width=.5pt] (13,-3)-- (16.20,-1.40);
	\draw [line width=.5pt] (15,-3)-- (14,1.39);
	\draw [fill=black] (13,-3) circle (5pt);
	\draw [fill=black] (15,-3) circle (5pt);
	\draw [fill=black] (16.20,-1.40) circle (5pt);
	\draw [fill=black] (15.80,0.50) circle (5pt);
	\draw [fill=black] (14,1.39) circle (5pt);
	\draw [fill=black] (12.20,0.50) circle (5pt);
	\draw [fill=black] (11.75,-1.40) circle (5pt);
\end{tikzpicture} \\

\begin{tikzpicture}[scale=.4]
	\draw [fill=black] (2,0) circle (5pt);
	\draw [fill=black] (1,1.732) circle (5pt);
	\draw [fill=black] (0,0) circle (5pt);
	\draw [fill=black] (-1,1.732) circle (5pt);
	\draw [fill=black] (-2,0) circle (5pt);
	\draw [fill=black] (-1,-1.732) circle (5pt);
	\draw [fill=black] (1,-1.732) circle (5pt);
	\draw [line width=.5pt] (1,1.732)-- (-1,1.732);
	\draw [line width=.5pt] (1,1.732)-- (2,0);
	\draw [line width=.5pt] (2,0)-- (1,-1.732);
	\draw [line width=.5pt] (1,-1.732)-- (-1,-1.732);
	\draw [line width=.5pt] (-1,-1.732)-- (-2,0);
	\draw [line width=.5pt] (-2,0)-- (-1,1.732);
	\draw [line width=.5pt] (0,0)-- (-1,1.732);
	\draw [line width=.5pt] (0,0)-- (1,1.732);
	\draw [line width=.5pt] (0,0)-- (-1,-1.732);
	\draw [line width=.5pt] (0,0)-- (1,-1.732);
	\draw [line width=.5pt] (1,1.732)-- (1,-1.732);
	\draw [line width=.5pt] (-1,1.732)-- (-1,-1.732);
\end{tikzpicture}
&&
\begin{tikzpicture}[scale=.4]	
	\draw [line width=.5pt] (-1,-10)-- (1,-10);
	\draw [line width=.5pt] (1,-10)-- (0,-8.3);
	\draw [line width=.5pt] (0,-8.27)-- (-1,-10);
	\draw [fill=black] (-1,-10) circle (5pt);
	\draw [fill=black] (1,-10) circle (5pt);
	\draw [fill=black] (0,-8.3) circle (5pt);
			\draw [line width=.5pt] (0,-8.3)-- (-1,-6.6);
	\draw [line width=.5pt] (0,-8.3)-- (1,-6.6);
	\draw [line width=.5pt] (-1,-6.6)-- (1,-6.6);
		\draw [fill=black] (-1,-6.6) circle (5pt);
	\draw [fill=black] (1,-6.6) circle (5pt);
\end{tikzpicture} &&
\begin{tikzpicture}[scale=.4]	
		\draw [line width=.5pt] (6,-10)-- (8,-10);
	\draw [line width=.5pt] (8,-10)-- (8.6,-8.1);
	\draw [line width=.5pt] (8.6,-8.1)-- (7,-6.90);
	\draw [line width=.5pt] (7,-6.90)-- (5.38,-8.10);
	\draw [line width=.5pt] (5.38,-8.10)-- (6,-10);
	\draw [line width=.5pt] (7,-6.90)-- (6,-10);
	\draw [line width=.5pt] (7,-6.90)-- (8,-10);
	\draw [line width=.5pt] (8.6,-8.1)-- (5.38,-8.10);
	\draw [line width=.5pt] (8.6,-8.1)-- (6,-10);
	\draw [line width=.5pt] (5.38,-8.10)-- (8,-10);
	\draw [fill=black] (6,-10) circle (5pt);
	\draw [fill=black] (8,-10) circle (5pt);
	\draw [fill=black] (8.6,-8.1) circle (5pt);
	\draw [fill=black] (7,-6.90) circle (5pt);
	\draw [fill=black] (5.38,-8.10) circle (5pt);
	\end{tikzpicture} &&
\begin{tikzpicture}[scale=.4]	
	\draw [line width=.5pt] (13,-10)-- (15,-10);
	\draw [line width=.5pt] (15,-10)-- (15.79,-7.8);
	\draw [line width=.5pt] (15.79,-7.8)-- (13.97,-6.5);
	\draw [line width=.5pt] (13.97,-6.5)-- (12.2,-7.9);
	\draw [line width=.5pt] (12.2,-7.9)-- (13,-10);
	\draw [fill=black] (13,-10) circle (5pt);
	\draw [fill=black] (15,-10) circle (5pt);
	\draw [fill=black] (15.79,-7.8) circle (5pt);
	\draw [fill=black] (13.97,-6.5) circle (5pt);
	\draw [fill=black] (12.2,-7.9) circle (5pt);
\end{tikzpicture}
\end{array}$ \\ \hline
\end{tabular}
\caption{Maximal graphs with rank up to  $5$.}
\label{tab:r2-5}
\end{table}

The maximal graphs up to rank $7$ were enumerated in \cite{ell} and independently in the series of the papers~\cite{tor1,tor2,lep,laz}.
More information on maximal graphs up to rank $7$  was given in ~\cite{laz} from which we quote
Tables~\ref{tab:r6}  and \ref{tab:t7} containing the distribution of maximal graphs with ranks $6$ and $7$ based on their orders.

\begin{table}[h!]
	\centering
		\begin{tabular}{rccccccccc}
		\hline
		Order & 6 & 7 & 8 & 9 & 10 & 11 & 12 & 13 & 14 \\ \hline
		$\#$ Maximal graphs & 5 & 0 & 2 & 5 & 2 & 2 & 6 & 2 & 3 \\
		\hline
	\end{tabular}
	\caption{The distribution of maximal graphs with rank 6.}
	\label{tab:r6}
\end{table}

\begin{table}[h!]
	\centering
	\begin{tabular}{rcccccccccccc}
		\hline
		Order & 7 & 8 & 9 & 10 & 11 & 12 & 13 & 14 & 15 & 16 & 17 & 18 \\ \hline
		$\#$ Maximal graphs & 13 & 4 & 18 & 2 & 32 & 13 & 63 & 11 & 19 & 5 & 0 & 3 \\
		\hline
	\end{tabular}
	\caption{The distribution of maximal graphs with rank 7.}
	\label{tab:t7}
\end{table}

We continue this line of work for the ranks 8 and 9. This is done by implementing an algorithm for constructing all maximal graphs with a given rank from \cite{ell, ack}.
For a given integer $r$, the input of the algorithm is the set of reduced graphs with both order and rank equal to $r$
and the output of the algorithm is the set of all maximal graphs of rank $r$.
The input of the algorithm was generated by using Mckay database of small graphs \cite{mk}. Consequently, we construct all maximal graphs with rank 8 and 9.
We found that there are exactly 2807 maximal graphs with rank 8. Their orders run over from 8 to 30. Also,
there are exactly 122511 maximal graphs with rank 9. Their orders run over from 9 to 38 with exceptions of 33, 35, 36.
In Table~\ref{tab:up to 9}, for the sake of completion, a summary of the number of maximal graphs of  rank up to 9 is given.
Moreover, the distributions of maximal graphs with rank 8 and 9 based on their orders are given in Tables~\ref{tab:rank8} and \ref{tab:r9}.
In Table~\ref{tab:r8edge}, we report more detailed information based on the orders and sizes (the number of edges) of maximal graphs with rank 8.
The Magma program and the data sets of maximal graphs with ranks $6,7,8,9$ is available online at \url{https://wp.kntu.ac.ir/ghorbani/comput}.

\begin{table}[h!]
	\centering
	\begin{tabular}{rcccccccc}
		\hline
		Rank & 2 & 3 & 4 & 5 & 6 & 7 & 8  &9\\ \hline
		$\#$ Maximal graphs & 1  & 1 & 3 & 8 & 27 & 183 & 2807& 122511 \\
		\hline
	\end{tabular}
	\caption{The number of maximal graphs with rank up to  9.}
	\label{tab:up to 9}
\end{table}

 \begin{table}[h!]
 	\centering
 	\begin{tabular}{rcccccccccccc}
 		\hline
 		Order & 8 & 9 & 10 & 11 & 12 & 13 & 14 & 15 & 16 & 17 & 18 & 19  \\ \hline
 		$\#$ Maximal graphs & 38 & 52 & 80 & 78 & 117 & 98 & 90 & 254 & 137 & 81 & 115 & 243 \\
 		\hline\vspace{-.25cm}\\  \hline
 		Order & 20 & 21 & 22 & 23 & 24 & 25 & 26 & 27 & 28 & 29 & 30 &    \\ \hline
 		$\#$ Maximal graphs & 884 & 252 & 134 & 69 & 57 & 7 & 7 & 5 & 3 & 2 & 4 &   \\
 		\hline
 	\end{tabular}
 	\caption{The distribution of maximal graphs with rank 8.}\label{tab:rank8}
 \end{table}

\begin{table}[h!]
	\centering
	\begin{tabular}{rcccccccccc}
		\hline
		Order  & 9 & 10 & 11 & 12 & 13 & 14 & 15 & 16 & 17 & 18  \\ \hline
		$\#$ Maximal graphs & 192 & 472 & 1014 & 786 &1402  & 1562 & 2198 & 1963 & 3509  & 2824  \\
		\hline \vspace{-.25cm}\\  \hline
		Order  & 19 & 20& 21 & 22 & 23 & 24 & 25 & 26 & 27 & 28  \\ \hline
		$\#$ Maximal graphs & 3660 & 17229 &  51315&20069  &8663  &2941 & 1622 & 528 &  266&136\\
		\hline\vspace{-.25cm}\\  \hline
		Order  & 29& 30&31  & 32& 33 & 34 &35 & 36 & 37& 38        \\ \hline
		$\#$ Maximal graphs &39& 42 & 42& 24 & 0 & 7& 0& 0 & 2 & 4
		\\
		\hline
	\end{tabular}
	\caption{The distribution of maximal graphs with rank 9.}\label{tab:r9}
\end{table}

\section*{Acknowledgements}
The research of the second author was partly funded by the Iranian National Science Foundation (INSF).
The authors would like to thank Max Alekseyev as the proof of Theorem~\ref{thm:nonMaximalF(k,m)} is essentially his argument posted on MathOverflow (\url{mathoverflow.net/questions/346085}) and also thank anonymous referees for constructive comments.

{}

\begin{table}
\centering
$${\tiny\begin{array}{|c|c|c| }
\hline	n & m & \# \\ \hline
	 & 12 & 1 \\
	 & 13 & 4 \\
	 & 14 & 3 \\
	 & 15 & 5 \\
	 & 16 & 6 \\
	 & 17 & 4 \\
	8 & 18 & 3 \\
	 & 19 & 5 \\
	 & 20 & 2 \\
	 & 21 & 1 \\
	 & 22 & 1 \\
	 & 23 & 1 \\
	 & 24 & 1 \\
	 & 28 & 1 \\ \hline
	 & 14 & 1 \\
	 & 15 & 5 \\
	 & 16 & 4 \\
	 & 17 & 4 \\
	 & 18 & 9 \\
	9 & 19 & 6 \\
	 & 20 & 7 \\
	 & 21 & 7 \\
	 & 22 & 4 \\
	 & 23 & 2 \\
	 & 24 & 2 \\
	 & 25 & 1 \\ \hline
	 & 19 & 3 \\
	 & 20 & 2 \\
	 & 21 & 5 \\
	 & 22 & 2 \\
	 & 23 & 6 \\
	 & 24 & 10 \\
	 & 25 & 8 \\
	 & 26 & 7 \\
	10 & 27 & 10 \\
	 & 28 & 4 \\
	 & 29 & 8 \\
	 & 30 & 7 \\
	 & 31 & 2 \\
	 & 32 & 1 \\
	 & 33 & 1 \\
	 & 34 & 1 \\
	 & 35 & 2 \\
	 & 39 & 1 \\ \hline
	 & 23 & 3 \\
	 & 24 & 5 \\
	 & 25 & 7 \\
	 & 26 & 13 \\
	 & 27 & 11 \\
	 & 28 & 8 \\
	 & 29 & 5 \\
	 & 30 & 2 \\
	 & 31 & 2 \\
	11 & 32 & 3 \\
	 & 33 & 4 \\
	 & 34 & 3 \\
	 & 35 & 3 \\
	 & 36 & 4 \\
	 & 37 & 1 \\
	 & 39 & 1 \\
	 & 41 & 1 \\
	 & 42 & 1 \\
	 & 43 & 1 \\ \hline
\end{array}\hspace{.17cm}
\begin{array}{|c|c|c| }
\hline	n & m & \# \\ \hline
	 & 27 & 3 \\
	 & 28 & 4 \\
	 & 29 & 11 \\
	 & 30 & 9 \\
	 & 31 & 9 \\
	 & 32 & 17 \\
	 & 33 & 7 \\
	& 34 & 9 \\
	 & 35 & 5 \\
	 & 36 & 7 \\
12 & 37 & 2 \\
	 & 38 & 9 \\
	 & 39 & 1 \\
	 & 40 & 3 \\
	 & 41 & 2 \\
	 & 42 & 8 \\
	 & 43 & 3 \\
	 & 44 & 3 \\
	 & 45 & 3 \\
	 & 46 & 1 \\
	 & 54 & 1 \\ \hline
	 & 34 & 2 \\
	 & 35 & 1 \\
	 & 36 & 4 \\
	 & 37 & 12 \\
	 & 38 & 9 \\
	 & 39 & 8 \\
	 & 40 & 10 \\
	 & 41 & 7 \\
	 & 42 & 4 \\
	 & 43 & 2 \\
	13 & 44 & 6 \\
	 & 45 & 9 \\
     & 46 & 5 \\
	 & 47 & 1 \\
	 & 48 & 7 \\
	 & 49 & 1 \\
	 & 51 & 2 \\
	 & 53 & 3 \\
	 & 54 & 2 \\
	 & 56 & 1 \\
	 & 60 & 2 \\\hline	
	 & 39 & 2 \\
	 & 41 & 1 \\
	 & 43 & 9 \\
	 & 44 & 4 \\
	 & 45 & 6 \\
	 & 46 & 2 \\
	 & 47 & 18 \\
	 & 48 & 7 \\
	 & 49 & 12 \\
	 & 50 & 5 \\
	14 & 51 & 9 \\
	 & 53 & 4 \\
	 & 54 & 1 \\
	 & 57 & 1 \\
	 & 58 & 2 \\
	 & 59 & 1 \\
	 & 61 & 1 \\
	 & 63 & 1 \\
	 & 64 & 1 \\
	 & 65 & 2 \\
	 & 67 & 1 \\ \hline
\end{array}\hspace{.17cm}
	\begin{array}{|c|c|c| }
\hline	n & m & \# \\ \hline
	 & 46 & 1 \\
	 & 47 & 2 \\
	 & 48 & 7 \\
	 & 49 & 4 \\
	 & 50 & 7 \\
	 & 51 & 11 \\
	 & 52 & 37 \\
	 & 53 & 29 \\
	 & 54 & 25 \\
	 & 55 & 17 \\
	 & 56 & 22 \\
	15 & 57 & 17 \\
	 & 58 & 10 \\
	 & 59 & 12 \\
	 & 60 & 13 \\
	 & 61 & 10 \\
	 & 62 & 4 \\
	 & 63 & 10 \\
	 & 64 & 7 \\
	 & 65 & 3 \\
	 & 67 & 2 \\
	 & 69 & 2 \\
	 & 71 & 1 \\
	 & 72 & 1 \\ \hline
	 & 52 & 1 \\
	 & 54 & 2 \\
	 & 56 & 3 \\
	 & 57 & 6 \\
	 & 58 & 4 \\
	 & 59 & 5 \\
	 & 60 & 10 \\
	 & 61 & 11 \\
	 & 62 & 19 \\
	 & 63 & 19 \\
	 & 64 & 18 \\
	 & 65 & 9 \\
	16 & 66 & 6 \\
	 & 67 & 8 \\
	 & 68 & 3 \\
	 & 69 & 1 \\
	 & 70 & 4 \\
	 & 72 & 1 \\
	 & 73 & 2 \\
	 & 74 & 1 \\
	 & 75 & 1 \\
	 & 78 & 2 \\
	 & 80 & 1 \\  \hline
	 & 55 & 1 \\
	 & 57 & 1 \\
	 & 58 & 1 \\
	 & 59 & 1 \\
	 & 60 & 2 \\
	 & 61 & 3 \\
	 & 62 & 3 \\
	17 & 63 & 1 \\
	 & 64 & 3 \\
	 & 65 & 8 \\
	 & 66 & 3 \\
	 & 67 & 6 \\
	 & 68 & 4 \\
	 & 69 & 3 \\
	 & 70 & 1 \\
 & 71 & 4 \\ \hline
\end{array}	\hspace{.17cm}
\begin{array}{|c|c|c| }
\hline	n & m & \# \\ \hline
	 & 72 & 5 \\
	 & 73 & 2 \\
	 & 74 & 3 \\
	 & 75 & 2 \\
	 & 76 & 5 \\
	 & 77 & 4 \\
	17 & 78 & 4 \\
	 & 79 & 2 \\
	 & 80 & 2 \\
	 & 81 & 1 \\
	 & 82 & 3 \\
	 & 84 & 1 \\
	 & 85 & 1 \\
	 & 88 & 1 \\ \hline
	 & 58 & 1 \\
	 & 60 & 1 \\
	 & 61 & 1 \\
	 & 62 & 6 \\
	 & 63 & 2 \\
	 & 64 & 9 \\
	 & 65 & 5 \\
	 & 66 & 17 \\
	 & 67 & 3 \\
	 & 68 & 12 \\
	 & 69 & 2 \\
	18 & 70 & 21 \\
	 & 71 & 4 \\
	 & 72 & 9 \\
	 & 73 & 4 \\
	 & 74 & 4 \\
	 & 76 & 4 \\
	 & 77 & 1 \\
	 & 80 & 1 \\
	 & 81 & 2 \\
	 & 84 & 1 \\
	 & 86 & 1 \\
	 & 87 & 3 \\
	 & 97 & 1 \\ \hline
	 & 59 & 1 \\
	 & 62 & 1 \\
	 & 63 & 1 \\
	 & 64 & 4 \\
	 & 65 & 1 \\
	 & 66 & 4 \\
	 & 67 & 3 \\
	 & 68 & 9 \\
	 & 69 & 8 \\
	 & 70 & 9 \\
	 & 71 & 14 \\
	19 & 72 & 23 \\
	 & 73 & 10 \\
	 & 74 & 17 \\
	 & 75 & 19 \\
	 & 76 & 25 \\
	 & 77 & 30 \\
	 & 78 & 10 \\
	 & 79 & 16 \\
	 & 80 & 9 \\
	 & 81 & 2 \\
	 & 82 & 4 \\
	 & 83 & 8 \\
	 & 84 & 7 \\
	 & 85 & 2 \\ \hline
\end{array}\hspace{.17cm}
\begin{array}{|c|c|c| }
\hline	n & m & \# \\ \hline
	 & 86 & 1 \\
	 & 87 & 1 \\
	19 & 91 & 1 \\
	 & 96 & 1 \\
	 & 99 & 1 \\
	 & 102 & 1 \\ \hline
	 & 70 & 3 \\
	 & 71 & 6 \\
	 & 72 & 8 \\
	 & 73 & 21 \\
	 & 74 & 18 \\
	 & 75 & 24 \\
	 & 76 & 44 \\
	 & 77 & 49 \\
	 & 78 & 48 \\
	 & 79 & 71 \\
	 & 80 & 80 \\
	 & 81 & 57 \\
	 & 82 & 82 \\
	 & 83 & 56 \\
	 & 84 & 62 \\
	20 & 85 & 54 \\
	 & 86 & 44 \\
	 & 87 & 37 \\
	 & 88 & 26 \\
	 & 89 & 34 \\
	 & 90 & 23 \\
	 & 91 & 11 \\
	 & 92 & 7 \\
	 & 93 & 3 \\
	 & 94 & 6 \\
	 & 95 & 2 \\
	 & 96 & 2 \\
	 & 97 & 1 \\
	 & 98 & 1 \\
	 & 100 & 2 \\
	 & 106 & 2 \\ \hline
	 & 76 & 2 \\
	 & 77 & 1 \\
	 & 78 & 8 \\
	 & 79 & 4 \\
	 & 80 & 6 \\
	 & 81 & 10 \\
	 & 82 & 8 \\
	 & 83 & 7 \\
	 & 84 & 11 \\
	 & 85 & 7 \\
	 & 86 & 19 \\
	 & 87 & 17 \\
	21 & 88 & 13 \\
	 & 89 & 19 \\
	 & 90 & 10 \\
	 & 91 & 8 \\
	 & 92 & 8 \\
	 & 93 & 14 \\
	 & 94 & 17 \\
	 & 95 & 9 \\
	 & 96 & 11 \\
	 & 97 & 12 \\
	 & 98 & 8 \\
	 & 99 & 6 \\
	 & 100 & 4 \\
	 & 101 & 2 \\ \hline
\end{array}\hspace{.17cm}
\begin{array}{|c|c|c| }
\hline	n & m & \# \\ \hline
	 & 102 & 4 \\
	 & 103 & 4 \\
	21 & 104 & 1 \\
	 & 106 & 1 \\
	 & 108 & 1 \\ \hline
	 & 83 & 1 \\
	 & 84 & 1 \\
	 & 85 & 3 \\
	 & 86 & 4 \\
	 & 87 & 2 \\
	 & 88 & 8 \\
	 & 89 & 2 \\
	 & 90 & 1 \\
	 & 91 & 7 \\
	 & 92 & 2 \\
	 & 93 & 3 \\
	 & 94 & 3 \\
	 & 95 & 10 \\
	 & 96 & 11 \\
	22 & 97 & 11 \\
	 & 98 & 10 \\
	 & 99 & 1 \\
	 & 100 & 3 \\
	 & 101 & 3 \\
	 & 102 & 2 \\
	 & 103 & 6 \\
	 & 104 & 9 \\
	 & 105 & 3 \\
	 & 106 & 5 \\
	 & 107 & 5 \\
	 & 108 & 2 \\
	 & 109 & 3 \\
	 & 110 & 2 \\
	 & 111 & 5 \\
	 & 112 & 1 \\
	 & 113 & 2 \\
	 & 114 & 1 \\
	 & 121 & 1 \\
	 & 127 & 1 \\ \hline
	 & 91 & 1 \\
	 & 93 & 1 \\
	 & 94 & 3 \\
	 & 95 & 2 \\
	 & 96 & 4 \\
	 & 97 & 3 \\
	 & 98 & 4 \\
	 & 99 & 1 \\
	 & 100 & 2 \\
	 & 101 & 2 \\
	23 & 102 & 2 \\
	 & 103 & 1 \\
	 & 104 & 2 \\
	 & 105 & 3 \\
	 & 106 & 2 \\
	 & 107 & 3 \\
	 & 108 & 1 \\
	 & 109 & 3 \\
	 & 110 & 1 \\
	 & 111 & 1 \\
	 & 113 & 4 \\
	 & 114 & 1 \\
	 & 115 & 3 \\
	 & 116 & 2 \\ \hline
\end{array}
\begin{array}{c}\begin{array}{|c|c|c|}
\hline	n & m & \# \\ \hline
	 & 117 & 1 \\
	 & 118 & 3 \\
	 & 119 & 4 \\
	 & 120 & 1 \\
	23 & 122 & 3 \\
	 & 123 & 1 \\
	 & 124 & 1 \\
	 & 125 & 1 \\
	 & 128 & 2 \\ \hline
	 & 98 & 2 \\
	 & 99 & 2 \\
	 & 100 & 1 \\
	 & 102 & 1 \\
	 & 103 & 2 \\
	 & 104 & 3 \\
	 & 106 & 5 \\
	 & 108 & 7 \\
	 & 110 & 1 \\
	 & 113 & 5 \\
	 & 115 & 1 \\
	24 & 117 & 1 \\
	 & 118 & 2 \\
	 & 119 & 1 \\
	 & 123 & 2 \\
	 & 124 & 4 \\
	 & 126 & 3 \\
	 & 127 & 1 \\
	 & 128 & 1 \\
	 & 129 & 1 \\
	 & 130 & 4 \\
	 & 133 & 2 \\
	 & 134 & 3 \\
	 & 136 & 1 \\
	 & 144 & 1 \\ \hline
	 & 106 & 1 \\
	 & 114 & 1 \\
	25 & 121 & 1 \\
	 & 122 & 1 \\
	 & 125 & 1 \\
	 & 145 & 2 \\ \hline
	 & 117 & 1 \\
	 & 125 & 1 \\
	26 & 133 & 3 \\
	 & 141 & 1 \\
	 & 169 & 1 \\ \hline
	 & 122 & 1 \\
	 & 150 & 1 \\
	27 & 151 & 1 \\
	 & 169 & 1 \\
	 & 170 & 1 \\ \hline
	 & 134 & 1 \\
	28 & 162 & 1 \\
	 & 196 & 1 \\   \hline
	29 & 142 & 1 \\
	 & 197 & 1 \\ \hline
	 & 155 & 1 \\
	30 & 187 & 1 \\
	 & 211 & 1 \\
	 & 225 & 1 \\  \hline
\end{array}
\\
\begin{array}{ccc}
&&\\\vspace{.5cm}
\end{array}
\end{array}
}$$
\caption{\footnotesize The distribution of maximal graphs with rank 8 in terms of order $n$ and size $m$.}\label{tab:r8edge}
\end{table}
\end{document}